\date{October 20, 2024}

\documentclass[10pt]{article}

\usepackage[small, pagestyles]{titlesec}
\usepackage[small,it]{caption}
\usepackage{enumitem}

\usepackage{nicefrac}

\usepackage[T1]{fontenc}

\usepackage{amsmath}
\usepackage{amsthm}
\usepackage{amssymb}

\usepackage[square,comma,numbers,sort&compress]{natbib}

\usepackage{tikz}

\usepackage{lineno}

\theoremstyle{plain}
\newtheorem{theorem}{Theorem}[section]
\newtheorem{proposition}[theorem]{Proposition}
\newtheorem*{proposition*}{Proposition}

\newfont{\footsc}{cmcsc10 at 8truept}
\newfont{\footbf}{cmbx10 at 8truept}
\newfont{\footrm}{cmr10 at 10truept}
\renewenvironment{abstract}%
                {
                  \begin{list}{}%
                     {\setlength{\rightmargin}{1in}%
                      \setlength{\leftmargin}{1in}}%
                   \item[]\ignorespaces\begin{small}}%
                 {\end{small}\unskip\end{list}}

\newpagestyle{main}[\small]{
        \headrule
        \sethead[\usepage][][]
        {\sc Bounds on the Lettericity of Graphs}{}{\usepage}}

\title{\sc Bounds on the Lettericity of Graphs}
\author{%
	Sean Mandrick and Vincent Vatter\\
	\small Department of Mathematics\\[-0.5ex]
	\small University of Florida\\[-0.5ex]
	\small Gainesville, Florida USA
}

\begin{document}
\maketitle

\begin{abstract}
Lettericity measures the minimum size of an alphabet needed to represent a graph as a letter graph, where vertices are encoded by letters, and edges are determined by an underlying decoder. We prove that all graphs on~$n$ vertices have lettericity at most approximately $n - \tfrac{1}{2} \log_2 n$ and that almost all graphs on $n$ vertices have lettericity at least $n - (2 \log_2 n + 2 \log_2 \log_2 n)$.
\end{abstract}

\section{Introduction}

Lettericity was first introduced by Petkov\v{s}ek~\cite{petkovsek:letter-graphs-a:} to investigate well-quasi-order in the induced subgraph order.
In Section~5.3 of his paper, Petkov\v{s}ek shows that there are $n$-vertex graphs with lettericity at least $0.707n$, and then Problem~3 of his conclusion asks to
``find the maximal possible lettericity of an $n$-vertex graph, and the corresponding extremal graphs.''
Despite significant recent interest in lettericity, both for its own sake~\mbox{\cite{%
	alecu:understanding-l:,
	ferguson:letter-graphs-a:,
	ferguson:on-the-letteric:,
	alecu:the-micro-world:,
	alecu:lettericity-of-:%
}},
and in its connections~\cite{%
	alecu:letter-graphs-a:iff,
	alecu:letter-graphs-a:,
	alecu:letter-graphs-a:abstract%
}
to geometric grid classes of permutations~\cite{%
	albert:geometric-grid-:,
	bevan:growth-rates-of:geom,
	albert:inflations-of-g:,
	elizalde:schur-positive-:%
}, this question has remained unaddressed until the present work. Our results demonstrate that the answer to the question is much greater than~$0.707n$. In particular, the greatest lettericity of an $n$-vertex graph lies between approximately $n-2\log_2 n$ and $n-\tfrac{1}{2}\log_2 n$.
We begin with some definitions.

For a finite alphabet $\Sigma$, we consider a set of ordered pairs $D \subseteq \Sigma^2$ which we refer to as a \textit{decoder}. Then for a word $w$ with letters $w(1)$, $w(2)$, $\dots$, $w(n)\in\Sigma$, we define the \textit{letter graph} of~$w$ with respect to~$D$ to be the graph $\Gamma_D(w)$ with the vertices $\{1,2,\dots,n \}$ and the edges $(i,j)$ for all $i<j$ with $(w(i), w(j)) \in D$.  If $|\Sigma| = k$ then we say that $\Gamma_D(w)$ is a $k$-letter graph. Finally, for any graph~$G$, the least integer $k$ such that~$G$ is (isomorphic to) a $k$-letter graph is called the \textit{lettericity} of~$G$, denoted by $\ell (G)$. 

We include some additional terminology here that will aid in the subsequent discussions. A word~$w$ is called a \textit{lettering} of a graph~$G$ if $\Gamma_D(w) = G$ for some decoder~$D$. We further say that each letter~$a\in\Sigma$ \textit{encodes} the vertices corresponding to the instances of $a$ in the word $w$. More precisely,~$a$ encodes the set $\{1\le i\le n: w(i) = a \} \subseteq V(\Gamma_D(w))$. The set of of vertices encoded by a given letter $a\in\Sigma$ must either form a clique, if $(a,a)\in D$, or an anticlique (independent set), if $(a,a)\notin D$. Thus letterings of graphs are special types of cocolorings (a concept introduced by Lesniak-Foster and Straight~\cite{lesniak-foster:the-cochromatic:}), and the lettericity of a graph is bounded below by its cochromatic number. However, as we will see, lettericity is typically much greater than cochromatic number.

A notable example of a class of graphs with lettericity 2 is the class of \emph{threshold graphs}~\cite{mahadev:threshold-graph:}. These graphs can be defined in various ways, but for our purposes, the most useful definition is as follows: a threshold graph is constructed by iteratively adding either dominating vertices (adjacent to all previously added vertices) or isolated vertices (adjacent to none of the previously added vertices). Thus, threshold graphs are precisely the letter graphs on the alphabet $\Sigma=\{a,b\}$ with the decoder $D=\{(a,b),(b,b)\}$. To see this, simply encode vertices based on their order of addition to the graph, using~$a$ for isolated vertices and~$b$ for dominating vertices.

Every $n$-vertex graph is an $n$-letter graph, as one can simply encode each vertex with its own letter and then add the appropriate pairs to the decoder. From this perspective, one `saves' letters by encoding multiple vertices with the same letter. It isn't difficult to see that the first and last vertices can always be encoded by the same letter, provided that no other vertices are encoded by that letter, and thus we can always save at least one letter. In other words, we have the elementary bound $\ell(G) \leq n-1$ for all $n$-vertex graphs~$G$. This gives rise to the following questions that we look to answer in this paper:
\begin{itemize}
	\item How many letters can we save in all graphs?
	\item How many letters can we expect to save in a random graph?
\end{itemize}

In Section 2, we use a Ramsey-type approach to show that we can save at least $k \approx \tfrac{1}{2} \log_2 n$ letters for every $n$-vertex graph, and thus the lettericity of every $n$-vertex graph is bounded above by approximately $n-\tfrac{1}{2} \log_2 n$. In Section~3, we show almost all $n$-vertex graphs have lettericity at least $n-(2\log_2 n+2\log_2\log_2 n)$.

Before getting to these results, the following proposition outlines the construction that will be used throughout the paper. We establish the upper bound in Section 2 by exploring ways to find induced subgraphs satisfying the hypotheses of Proposition~\ref{mainprop}. Then in Section 3, we will see that for almost all graphs, the only way to save letters is by finding induced subgraphs that satisfy these hypotheses.

\begin{proposition}
\label{mainprop}
Suppose~$G$ is a graph with $n$ vertices containing an induced subgraph~$H$ with $2k$ vertices that is a $k$-letter graph for a word of the form
\[
	w = \ell_1 \hspace{1mm} \ell_2 \dots \ell_k \hspace{1mm} \ell_{\pi(1)} \hspace{1mm} \ell_{\pi(2)} \dots \ell_{\pi(k)}
\]
for some permutation $\pi$ of $\{1, \dots, k \}$. Then, $w$ can be extended to a lettering of~$G$ by inserting new letters into the middle of $w$, and thus $\ell(G) \leq n-k$. 
\end{proposition}

\begin{proof}
Suppose~$H$ and $w$ are as in the hypothesis and that $D_1 \subseteq \{\ell_1, \ell_2, \dots, \ell_k \}^2$ is a decoder for which
\[
\Gamma_{D_1}(w) = H.
\]

Label the vertices of $G-H$ as $v_1,v_2,\dots,v_{n-2k}$ and let $\lambda_1$, $\dots$, $\lambda_{n-2k}$ be a set of distinct new letters disjoint from $\ell_1$, $\dots$, $\ell_k$. By choosing as our decoder the set $D_2 = \{ (\lambda_i, \lambda_j) : 1 \leq i < j \leq n-2k \textnormal{ and } v_i v_j \in E(G) \}$, we see immediately that
\[
\Gamma_{D_2}(\lambda_1 \lambda_2 \dots \lambda_{n-2k}) = G-H.
\]
Next define the word 
\[
w' = \ell_1 \dots \ell_k \hspace{1mm} \lambda_1 \lambda_2 \dots \lambda_{n-2k} \hspace{1mm} \ell_{\pi(1)} \dots \ell_{\pi(k)},
\]
and for each $1\le i\le k$, let $x_i$ and $y_i$ be the vertices encoded by the left and right instances of $\ell_i$ in~$w$, respectively. Now define the sets 
\begin{align*}
    D_x & = \{(\ell_i, \lambda_j):
    	\text{$1\le i\le k$, $1\le j\le n-2k$, and $x_i v_j \in E(G)$}\},\\
    D_y & = \{(\lambda_j, \ell_i):
    	\text{$1\le i\le k$, $1\le j\le n-2k$, and $v_j y_i \in E(G)$}\}.
\end{align*}
Letting $D = D_1 \cup D_2 \cup D_x \cup D_y$, it follows that $\Gamma_D(w') = G$, which proves the result.
\end{proof}

\section{Saving letters in all graphs}
\label{sec:upper-bound}

One way to satisfy the hypothesis of Proposition~\ref{mainprop} is to take~$H$ to be a clique or anticlique. Letting $R(k)$ denote the $k$th diagonal Ramsey number, we know that every graph on at least~$R(2k)$ vertices has such a subgraph, and so we obtain the following.

\begin{proposition}
\label{ramseyprop}
For each $k$ and any graph~$G$ on $n \geq R(2k)$ vertices,~$G$ has an induced subgraph with $2k$ vertices that is a $k$-letter graph on the word
\[
	w = \ell_1 \hspace{1mm} \ell_2 \dots \ell_k \hspace{1mm} \ell_{\pi(1)} \hspace{1mm} \ell_{\pi(2)} \dots \ell_{\pi(k)}
\]
for \emph{any} permutation $\pi$ of $\{1, \dots, k \}$. Thus, $\ell(G) \leq n - k$ by Proposition~\ref{mainprop}.
\end{proposition}


As it is known that
\[
	\sqrt{2}^k < R(k) \leq 4^k
\]
for all $k$, Proposition~\ref{ramseyprop} implies that for all graphs on $n \approx 4^{2k}$ vertices, we can save $k \approx \tfrac{1}{4} \log_2 n$ letters. We show below that we can save twice this many letters in every $n$-vertex graph.

\begin{theorem}
\label{mainconstruction}
For every $k$ and each graph~$G$ on $n \geq 2(k-1) + 2^{2(k-1)} + 1$ vertices,~$G$ has an induced subgraph with $2k$ vertices that is a $k$-letter graph on the word
\[
w = \ell_1 \hspace{1mm} \ell_2 \dots \ell_k \hspace{1mm} \ell_k \dots \ell_2 \hspace{1mm} \ell_1.
\]
Thus, $\ell(G) \leq n - k$ by Proposition~\ref{mainprop}.
\end{theorem}

\begin{proof}
We use induction on $k$. For the base case of $k=1$, we have a graph~$G$ on $n\ge 2$ vertices, and the desired induced subgraph can be obtained by taking any two vertices.

Now suppose the result holds for some $k \geq 1$ and that~$G$ is a graph on $n \geq 2k + 2^{2k} + 1$ vertices. By our hypotheses, we have that~$G$ has an induced subgraph~$H$ that is a $k$-letter graph on the word $w = \ell_1 \hspace{1mm} \ell_2 \dots \ell_k \hspace{1mm} \ell_k \dots \ell_2 \hspace{1mm} \ell_1$, say with decoder $D$. Since there are $2^{2k} + 1$ vertices in~$G$ that are not in~$H$, the pigeonhole principle tells us that two of these vertices, call them $u$ and $v$, must agree on all of the vertices in~$H$. Let $H'$ be the induced subgraph of~$G$ on the vertex set $V(H) \cup \{ u,v \}$. 

We claim that $H'$ is a letter graph on the word $w' = \ell_1 \hspace{1mm} \ell_2 \dots \ell_k \hspace{1mm} \ell_{k+1} \hspace{1mm} \ell_{k+1} \hspace{1mm} \ell_k \dots \ell_2 \hspace{1mm} \ell_1$. For each ${1\le i\le k}$, let $x_i$ denote the vertex in~$H$ that is encoded by the left occurrence of $\ell_i$ in $w$, and similarly, let $y_i$ be the vertex that is encoded by the right occurrence of $\ell_i$, (just as in the proof of Proposition~\ref{mainprop}). Now let $X = \{ (\ell_i, \ell_{k+1}): x_i u \in E(G), (\textnormal{equivalently } x_i v \in E(G)) \}$ and $Y = \{ (\ell_{k+1}, \ell_i): u y_i \in E(G), (\textnormal{equivalently } v y_i \in E(G)) \}$. Next, let~$Z$ be the set $\{(\ell_{k+1}, \ell_{k+1}) \}$ if ${uv \in E(G)}$ and~$\varnothing$ otherwise. Then it follows that~$H'$ is a letter graph on the word~$w'$ with decoder $D' = D \cup X \cup Y \cup Z$, that is, $\Gamma_{D'} (w') = H'$. This gives the result.
\end{proof}

\section{Failing to save letters in almost all graphs}
\label{sec:lower-bound}

We now focus on demonstrating that almost all graphs have large lettericity, indicating that there is little room for improvement on the upper bound given in Theorem~\ref{mainconstruction}. Recall that the random graph $G(n, \nicefrac{1}{2})$ is the graph on $n$ vertices where each edge appears independently with probability $\nicefrac{1}{2}$. Thus, every labeled $n$-vertex graph occurs with equal probability. We show that with probability tending to $1$ as $n\to\infty$, the lettericity of $G(n, \nicefrac{1}{2})$ is at least $n-(2\log_2 n + 2\log_2\log_2 n)$. First, we prove two results that greatly restrict the possible letterings of almost all graphs.

\begin{proposition}\label{threeverts}
For almost all graphs~$G$, no three vertices can be encoded by the same letter in a lettering of~$G$. 
\end{proposition}

\begin{proof}
Letting $G=G(n, \nicefrac{1}{2})$, we show that the probability that three vertices can be encoded with the same letter in a lettering of~$G$ tends to $0$ as $n \to \infty$. Assume that we have a lettering $w$ of~$G$ using the alphabet $\Sigma$, and that three vertices are encoded by the letter $a \in \Sigma$. Then there exist four possibly empty words $w_1$, $w_2$, $w_3$, and $w_4$, such that
	\[
	w = w_1 \hspace{1mm} a \hspace{1mm} w_2 \hspace{1mm} a \hspace{1mm} w_3 \hspace{1mm} a \hspace{1mm} w_4.
	\]

Let $x$, $y$ and $z$ be the vertices encoded by the left, middle and right instances of $a$ in~$w$, respectively. If a vertex is encoded by the instance of some letter in $w_1$ or $w_4$, then it must agree on each of the vertices $x,y$ and $z$. If a vertex is encoded in one of $w_2$ or $w_3$, then it either must agree on $y$ and $z$, or on $x$ and $y$, respectively.

For any vertex $v \in V(G) \setminus \{x,y,z\}$, there are four possible ways it can agree or disagree with~$x$,~$y$, and~$z$: it either agrees on all three vertices, agrees on~$\{x,y\}$ but not on~$z$, agrees on~$\{y,z\}$ but not on~$x$, or agrees on~$\{x,z\}$ but not on~$y$. Since $G=G(n, \nicefrac{1}{2})$, these four possibilities are equally likely, and only the last option prevents~$w$ from having a place in which $v$ can be encoded. Thus, the probability that~$v$ can be encoded in~$w$ is $\nicefrac{3}{4}$. It follows that the probability that every vertex in $V(G)\setminus\{x,y,z\}$ can be encoded in $w$ is~$(\nicefrac{3}{4})^{n-3}$.
	
Now let $A_{(x,y,z)}$ be the event that the vertices $x$, $y$ and $z$  can be encoded, in that order, by the same letter in a lettering of~$G$. We see from above that $\Pr[A_{(x,y,z)}] \leq (\nicefrac{3}{4})^{n-3}$. (In fact, the probability is~$0$ if $x$, $y$ and $z$ do not form a clique or anticlique.) Next, define the event
\[
	A = \bigcup_{(x,y,z)} A_{(x,y,z)},
\]
where the union is taken over all sequences $(x,y,z)$ of distinct vertices of~$G$. Thus, $A$ is the event that any three vertices can be encoded by the same letter in a lettering of~$G$. We have that
\[
	\Pr[A] \leq \sum_{(x,y,z)} \Pr[A_{(x,y,z)}] \leq n(n-1)(n-2) \cdot (\nicefrac{3}{4})^{n-3}, 
\]
and therefore $\Pr[A]\to 0$ as $n \to \infty$.
\end{proof}

Because of this result, we may henceforth assume that we are not able to encode three or more vertices using the same letter. As a consequence, if we are to save letters, we must do so in pairs. With this assumption, we see next that in any lettering of almost every graph~$G$, the letters that appear in pairs are crossing or nested. That is, for almost all graphs~$G$, if the letters $a,b \in \Sigma$ both appear twice in a lettering $w$ of~$G$, then it never happens that they appear \textit{separated} as $\cdots a\cdots a\cdots b\cdots b\cdots$, but rather they must appear \emph{crossing} as $\cdots a\cdots b\cdots a\cdots b\cdots$ or \emph{nested} as $\cdots a\cdots b\cdots b\cdots a\cdots$.

\begin{proposition}\label{prop:crossornest}
For almost all graphs~$G$, if two letters appear twice in a lettering of~$G$, they must appear in a crossing or nested pattern.
\end{proposition}

\begin{proof}
Letting $G=G(n, \nicefrac{1}{2})$, we show that the probability that~$G$ has a lettering in which two pairs of letters appear in a separated pattern tends to~$0$ as $n \to \infty$. This will yield the result since the only other possibility is that the pairs of letters are crossing or nested. Suppose that we have a lettering~$w$ of~$G$ over the alphabet $\Sigma$ containing $a$ and $b$ given by
\[
	w = w_1 \hspace{1mm} a \hspace{1mm} w_2 \hspace{1mm} a \hspace{1mm} w_3 \hspace{1mm} b \hspace{1mm} w_4 \hspace{1mm} b \hspace{1mm} w_5,
\]
for some possibly empty words $w_1$, $w_2$, $w_3$, $w_4$, and $w_5$. Let $x$, $y$, $s$, and~$t$ be the vertices of~$G$ encoded in $w$ by these instances of $a$ and $b$, reading left to right.
Fix a vertex in $V(G) \setminus \{x,y,s,t\}$.
This vertex can be encoded in $w_1$, $w_3$, or $w_5$ only if it agrees on~$\{x,y\}$ and on~$\{s,t\}$. Further, it can be encoded in~$w_2$ or~$w_4$ only if it agrees on~$\{s,t\}$ or on~$\{x,y\}$, respectively.

For each vertex $v \in V(G) \setminus \{x, y, s, t\}$, there are four possible ways it can agree or disagree on the pairs~$\{x,y\}$ and~$\{s,t\}$: it either agrees on both pairs, agrees only on~$\{x,y\}$, agrees only on~$\{s,t\}$, or disagrees on both pairs. These possibilities are equally likely because $G = G(n, \nicefrac{1}{2})$, and only the last case prevents~$w$ from having a place in which~$v$ can be encoded. Thus, the probability that~$v$ can be encoded somewhere in~$w$ is~$\nicefrac{3}{4}$. Hence, the probability that every vertex in $V(G) \setminus \{x,y,s,t\}$ can be encoded in~$w$ is $(\nicefrac{3}{4})^{n-4}$.

Let $B_{(x,y,s,t)}$ be the event that there is a lettering of~$G$ in which the vertices $x$, $y$, $s$ and $t$ are encoded in that order, $x$ and $y$ are encoded by the same letter, and $s$ and $t$ are encoded by a second letter. Thus, this is the event that these four vertices can correspond to a separated pattern encoded in the given order. From above, we have that $\Pr[B_{(x,y,s,t)}] \le (\nicefrac{3}{4})^{n-4}$. (In fact, this probability is $0$ if $x$ and $y$ do not agree on $s$ and $t$, and vice versa.) Next, define the event
\[
B = \bigcup_{(x,y,s,t)} B_{(x,y,s,t)},
\]
where the union is over all sequences $(x,y,s,t)$ of distinct vertices of~$G$. Thus,~$B$ is the event that a lettering of~$G$ has two pairs of letters in a separated pattern. We have that 
\[
\Pr[B] \leq \sum_{(x,y,s,t)} \Pr[B_{(x,y,s,t)}] \le n(n-1)(n-2)(n-3) \cdot (\nicefrac{3}{4})^{n-4}, 
\]
and therefore $\Pr[B]\to 0$ as $n \to \infty$.
\end{proof}

By Propopsition~\ref{threeverts}, we know that for almost all graphs~$G$ on $n$ vertices, if~$G$ has a lettering with~${n-k}$ letters, then it will have $k$ letters that appear twice and $n-2k$ letters that appear once. Suppose that the letters appearing twice are $\ell_1$, $\ell_2$, $\dots$, $\ell_k$. By Proposition~\ref{prop:crossornest}, we know that for almost all graphs that have such a lettering, there is a permutation~$\pi$ of $\{1, \dots, k \}$ such that the subword of $w$ containing all of these letters is
\begin{equation}\label{word}\tag{$\dagger$}
    \ell_1 \hspace{1mm} \ell_2 \dots  \ell_k \hspace{1mm} \ell_{\pi(1)} \ell_{\pi(2)} \dots \ell_{\pi(k)}.
\end{equation}
Note that this is the same construction considered in Proposition~\ref{mainprop}.

It remains to analyze the probability that there is an induced subgraph that can be lettered by a word such as that of (\ref{word}). To this end, let~$G = G(n, \nicefrac{1}{2})$ and $k$ an integer satisfying $2k\le n$. Further, let $(v_i)=(v_1,\dots,v_{2k})$ be a sequence of distinct vertices of~$G$ and $\pi$ a permutation of $\{1,\dots,k\}$. We define $C_{(v_i),\pi}$ to be the event that there exists a decoder $D\subseteq\{\ell_1,\dots,\ell_k\}^2$ such that the mapping $v_i\mapsto i$ is an isomorphism between the induced subgraph $G[\{v_1,\dots,v_{2k}\}]$ and the letter graph $\Gamma_D(\ell_1\cdots\ell_k \hspace{1mm} \ell_{\pi(1)}\cdots\ell_{\pi(k)})$.

To evaluate $\Pr[C_{(v_i), \pi}]$, for every pair $i < j$ we handle the case of the $\ell_i$'s and $\ell_j$'s being either crossing or nested in the word $\ell_1\cdots\ell_k\hspace{1mm}\ell_{\pi(1)}\cdots\ell_{\pi(k)}$. If these letters are crossing, then, as is indicated by the three lines in Figure~\ref{fig:cross}, there are three potential edges that must agree. That is, all three of these edges will be decided by the presence or absence of $(\ell_i, \ell_j)$ in the decoder, and hence they must all be edges or non-edges. With each edge decided with probability $\nicefrac{1}{2}$, the probability that the three potential edges agree is $\nicefrac{1}{4}$.

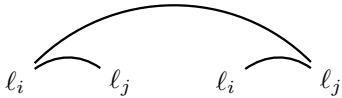
\begin{figure}[ht]
    \centering
    \begin{tikzpicture}[scale=0.7]
        \node (w) at (0,0) {$\ell_i$};
        \node (x) at (2,0) {$\ell_j$};
    
        \node (y) at (4,0) {$\ell_i$};
        \node (z) at (6,0) {$\ell_j$};

        \draw [line width = 0.3mm]   (w) to[out=45,in=135] (z);
        \draw [line width = 0.3mm]   (w) to[out=35,in=145] (x);
        \draw [line width = 0.3mm]   (y) to[out=35,in=145] (z);
        
       \end{tikzpicture}
       \caption{The potential edges that must agree in a crossing pattern.}
       \label{fig:cross}
\end{figure}

If the letters $\ell_i$ and $\ell_j$ are nested, then we see in Figure \ref{fig:nest} that there are two pairs of potential edges that must agree. That is, the solid lines must agree since they are both determined by the presence or absence of $(\ell_i, \ell_j)$ in the decoder, and the dotted lines must agree because they are both determined by the presence or absence of $(\ell_j, \ell_i)$ in the decoder. Again, with each of these edges decided with probability $\nicefrac{1}{2}$, the probability that both of these pairs of potential edges agree is $\nicefrac{1}{4}$. 

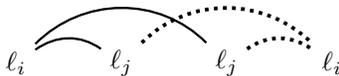
\begin{figure}[ht]
    \centering
    \begin{tikzpicture}[scale=0.7]
        \node (w) at (0,0) {$\ell_i$};
        \node (x) at (2,0) {$\ell_j$};
    
        \node (y) at (4,0) {$\ell_j$};
        \node (z) at (6,0) {$\ell_i$};
    
        \draw [line width = 0.3mm]   (w) to[out=45,in=135] (y);
        \draw [line width = 0.3mm]   (w) to[out=35,in=145] (x);
        \draw [line width = 0.5mm, dotted]   (y) to[out=35,in=145] (z);
        \draw [line width = 0.5mm, dotted]   (x) to[out=45,in=135] (z);
        \end{tikzpicture}
        \caption{The potential edges that must agree in a nested pattern.}
    \label{fig:nest}
\end{figure}

For each pair $i<j$, which of these two cases must be satisfied is determined by $\pi$, and since each case has probability~$\nicefrac{1}{4}$ of being satisfied, it follows that
\[
	\Pr[C_{(v_i), \pi}] = (\nicefrac{1}{4})^{\binom{k}{2}} = 2^{-k(k-1)}.
\]
Next, we define the event
\[
	C = \bigcup_{(v_i), \pi} C_{(v_i), \pi},
\]
where the union is over all sequences $(v_i)$ of $2k$ distinct vertices of~$G$ and all permutations $\pi$ of $\{1, \dots, k \}$. Thus, $C$ is the event that an induced subgraph of~$G$ can be encoded in the form of (\ref{word}). In light of the preceding arguments, we can regard $C$ as the only event in which $k$ letters can be saved in a lettering of almost all graphs~$G$. To obtain the main result of this section, we simply need to minimize the value of $k$ so that the probability of $C$ still goes to zero as $n\to\infty$.

\begin{theorem}
	For almost all graphs~$G$ with $n$ vertices, we have 
	\[
	\ell(G) \geq n - (2 \log_2 n + 2 \log_2 \log_2 n).
	\]
\end{theorem}

\begin{proof}
It is clear from above that
\begin{equation*}
	\Pr[C]
	\leq
	\sum_{(v_i), \pi} \Pr[C_{(v_i), \pi}]
	=
	n(n-1)\cdots (n-2k+1) \cdot k! \cdot 2^{-k(k-1)}.
\end{equation*}
Using the straightforward inequalities $n(n-1)\cdots (n-2k+1)\le n^{2k}$ and $k! \leq k^k$, we see that
\[
	\Pr[C]
	\leq
	n^{2k} k^k 2^{-k(k-1)}
	=
	(n^2 k 2^{-k+1})^{k}.
\]
Setting $k = 2 \log_2 n + 2 \log_2 \log_2 n$, simple computations show that $\Pr[C]$ tends to $0$ as $n \to \infty$, and therefore the result follows from the preceding arguments.
\end{proof}

%
%
%


\begin{thebibliography}{10}

\bibitem{albert:geometric-grid-:}
{\sc Albert, M., Atkinson, M., Bouvel, M., Ru{\v{s}}kuc, N., and Vatter, V.}
\newblock Geometric grid classes of permutations.
\newblock {\em Trans. Amer. Math. Soc. 365}, 11 (2013), 5859--5881.

\bibitem{albert:inflations-of-g:}
{\sc Albert, M., Ru{\v{s}}kuc, N., and Vatter, V.}
\newblock Inflations of geometric grid classes of permutations.
\newblock {\em Israel J. Math. 205}, 1 (2015), 73--108.

\bibitem{alecu:letter-graphs-a:iff}
{\sc Alecu, B., Ferguson, R., Kant{\'e}, M., Lozin, V., Vatter, V., and
  Zamaraev, V.}
\newblock Letter graphs and geometric grid classes of permutations.
\newblock {\em SIAM J. Discrete Math. 36}, 4 (2022), 2774--2797.

\bibitem{alecu:lettericity-of-:}
{\sc Alecu, B., Kant{\'e}, M., Lozin, V., and Zamaraev, V.}
\newblock Lettericity of graphs: an {FPT} algorithm and a bound on the size of
  obstructions.
\newblock arXiv:2402.12559 [math.CO].

\bibitem{alecu:understanding-l:}
{\sc Alecu, B., and Lozin, V.}
\newblock Understanding lettericity {I}: a structural hierarchy.
\newblock arXiv:2106.03267 [math.CO].

\bibitem{alecu:the-micro-world:}
{\sc Alecu, B., Lozin, V., and de~Werra, D.}
\newblock The micro-world of cographs.
\newblock In {\em Combinatorial Algorithms}, L.~G\k{a}sieniec, R.~Klasing, and
  T.~Radzik, Eds., vol.~12126 of {\em Lecture Notes in Comput. Sci.} Springer,
  Cham, Switzerland, 2020, pp.~30--42.

\bibitem{alecu:letter-graphs-a:}
{\sc Alecu, B., Lozin, V., de~Werra, D., and Zamaraev, V.}
\newblock Letter graphs and geometric grid classes of permutations:
  characterization and recognition.
\newblock {\em Discrete Appl. Math. 283\/} (2020), 482--494.

\bibitem{alecu:letter-graphs-a:abstract}
{\sc Alecu, B., Lozin, V., Zamaraev, V., and de~Werra, D.}
\newblock Letter graphs and geometric grid classes of permutations:
  characterization and recognition.
\newblock In {\em Combinatorial Algorithms}, L.~Brankovic, J.~Ryan, and W.~F.
  Smyth, Eds., vol.~10765 of {\em Lecture Notes in Comput. Sci.} Springer,
  Cham, Switzerland, 2018, pp.~195--205.

\bibitem{bevan:growth-rates-of:geom}
{\sc Bevan, D.}
\newblock Growth rates of geometric grid classes of permutations.
\newblock {\em Electron. J. Combin. 21}, 4 (2014), Paper 4.51, 17 pp.

\bibitem{elizalde:schur-positive-:}
{\sc Elizalde, S., and Roichman, Y.}
\newblock Schur-positive sets of permutations via products and grid classes.
\newblock {\em J. Algebraic Combin. 45}, 2 (2017), 363--405.

\bibitem{ferguson:on-the-letteric:}
{\sc Ferguson, R.}
\newblock On the lettericity of paths.
\newblock {\em Australas. J. Combin. 78}, 2 (2020), 348--351.

\bibitem{ferguson:letter-graphs-a:}
{\sc Ferguson, R., and Vatter, V.}
\newblock Letter graphs and modular decomposition.
\newblock {\em Discrete Appl. Math. 309\/} (2022), 215--220.

\bibitem{lesniak-foster:the-cochromatic:}
{\sc Lesniak-Foster, L.~M., and Straight, H.~J.}
\newblock The cochromatic number of a graph.
\newblock {\em Ars Combin. 3\/} (1977), 39--45.

\bibitem{mahadev:threshold-graph:}
{\sc Mahadev, N. V.~R., and Peled, U.~N.}
\newblock {\em Threshold Graphs and Related Topics}, vol.~56 of {\em Annals of
  Discrete Mathematics}.
\newblock North-Holland, Amsterdam, The Netherlands, 1995.

\bibitem{petkovsek:letter-graphs-a:}
{\sc Petkov{\v{s}}ek, M.}
\newblock Letter graphs and well-quasi-order by induced subgraphs.
\newblock {\em Discrete Math. 244}, 1-3 (2002), 375--388.

\end{thebibliography}

\begin{small}

\end{small}

\end{document}